\documentclass[preprint,review,10pt]{amsart}
\usepackage{amsmath,amssymb,amsfonts,xspace,mathrsfs}
\newtheorem{theorem}{Theorem}[section]
\newtheorem{lemma}[theorem]{Lemma}

\theoremstyle{definition}
\newtheorem{definition}[theorem]{Definition}
\newtheorem{example}[theorem]{Example}

\newtheorem{proposition}[theorem]{Proposition}
\newtheorem{corollary}[theorem]{Corollary}

\theoremstyle{remark}
\newtheorem{remark}[theorem]{Remark}

\numberwithin{equation}{section}

\begin{document}

\title{$S$-filters of bounded lattices}

\author{Mahdi Anbarloei}
\address{Department of Mathematics, Faculty of Sciences,
Imam Khomeini International University, Qazvin, Iran.
}

\email{m.anbarloei@sci.ikiu.ac.ir}





\begin{abstract}
 In this paper, we introduce and study the notion of  $S$-filters in  bounded distributive lattices.

\end{abstract}
\maketitle

\section{Introduction} 
We assume throughout this paper that all lattices are bounded, meaning they have a least element $0$ and a greatest elemen $1$. Considering lattices as generalizations of rings, it is natural to examine which ring properties remain valid in lattice theory. Since the absence of subtraction prevents many ring results from having direct counterparts, filling this conceptual gap remains a heavily pursued objective in the literature.

Prime and primary ideals play a pivotal role in ring theory, inspiring numerous papers on their concepts and extensions. However, the research paradigm for generalizing prime ideals shifted significantly in 2020, when Hamed and Malek  introduced the notion of $S$-prime ideals in a commutative ring $A$ where $S$ is a multiplicatively closed subset of $A$.  Recently, The notable concept of $S$-ideals, defined with respect to a multiplicatively closed subset, has recently been introduced in \cite{Khashan}.   A  proper ideal $J$  of $A$ is said to be   an $S$-ideal if for all $x_1,x_2 \in A$, $x_1x_2 \in J$ and $x_1 \in S$ imply that $x_2 \in J$ \cite{Khashan}. 

In \cite{Atani2}, while defining a $\vee$-closed subset of a lattice $\mathscr{L}$, Atani extended the concept  of $S$-primeness from
a commutative ring to $S$-primeness in a lattice. This extension was called $S$-prime filters. Moreover, the notions of $S$-2-absorbing   and weakly $S$-2-absorbing filters in lattices were introduced in \cite{Atani3} and \cite{Atani5}, respectively.

In this paper, we aim to introduce the notion of $S$-filters in a bounded distributive lattice $\mathscr{L}$ as an extension of $S$-ideal property in commutative rings,  offering a new perspective on the classification of filters by means of a  $\vee$-closed subset $S$ in $\mathscr{L}$. Among the various findings established in this work,   we construct the smallest $S$-filter of $\mathscr{L}$ containing a given filter $\mathfrak{p}$ in in Theorem \ref{small}. Theorem \ref{3} provides an $S$-filter analogue of the Prime Avoidance Lemma in Lattices. It is proved in Proposition \ref{intersection} that the class of   $S$-filters is closed under arbitrary intersections. Theorem \ref{khamen} indicates that a filter $F$ of $\mathscr{L}$ disjoin with $S$ can be extended to an $S$-filter of $\mathscr{L}$. Finally, we examine the  behavior of $S$-filters under homomorphis and in Cartesian products of lattices in Theorem \ref{homo} and Theorem \ref{car}, respectively.

\section{Preliminaries}
In this section, we briefly recall some preliminary definitions and basic results  from \cite{Birkhoff} and other research works that will be utilized throughout this paper.

A poset $(\mathscr{L}, \le)$ is a  \textit{lattice}  if   each pair $u, v \in \mathscr{L}$ has a  greatest lower bound (briefly, g.l.b.) and a   least upper bound (briefly, l.u.b.), denoted by $u \wedge v$ and $u \vee v$, respectively. 
A lattice $\mathscr{L}$ is called  \textit{complete}  if every subset of $\mathscr{L}$ has a greatest lower bound and least upper bound in $\mathscr{L}$. It follows immediately that every nonvoid complete lattice contains a least element $0$ and a greatest element $1$. In this case, $\mathscr{L}$ is a bounded lattice.

Let $\mathscr{L}$ be a lattice. A non-empty subset  $F$ of $\mathscr{L}$ is said to be a \textit{filter}, if $u \wedge v \in F$ for all $u, v \in F$, and for $u \in F$ and $w \in \mathscr{L}$, $u \le w$ implies $w \in F$. Assume that $\mathscr{L}$ is a lattice with $1$. Then,  $1$ is an element of   every filter of $\mathscr{L}$ and $\{1\}$ is a filter of $\mathscr{L}$. By $\mathcal{F}(\mathscr{L})$, we mean the set of all filters of $\mathscr{L}$. Let $u \in \mathscr{L}$.  A complement of $u$ in $\mathscr{L}$ is an element $v \in \mathscr{L}$ such that $u \vee v = 1$ and $u \wedge v= 0$.   $\mathscr{L}$ is called a complemented lattice if any element of $\mathscr{L}$ has a complement in $\mathscr{L}$.

\begin{definition}
\cite{Birkhoff} Let  $F$ be a proper filter of $\mathscr{L}$. We say that 
\begin{enumerate}
\item $F$ is \textit{prime} if $u \vee v \in F$ implies that $u \in F$ or $v \in F$.
\item  $F$ is  \textit{maximal} if whenever $H$ is a filter in $\mathscr{L}$ such that $F \subsetneqq H$, then $H = \mathscr{L}$.
\end{enumerate}
\end{definition}

\begin{definition} \cite{Atani1}
Let $\mathscr{L}$ be a lattice and $X \subseteq \mathscr{L}$.  The filter generated by $X$, denoted by $\text{T}(A)$,  is defined as 
$$\text{T}(X) = \bigcap \{ F \in  \mathcal{F}(\mathscr{L}) \mid   X \subseteq F \}.$$
\end{definition}

A filter $F$ of a lattice $\mathscr{L}$ is   finitely generated if $F = T(X)$ for some finite subset $X$ of $\mathscr{L}$.
\begin{lemma} \cite{Atani1}
Let $X $ be a non-empty of  $\mathscr{L}$. Then, \[T(X) = \{u \in \mathcal{L} \mid  u_1 \wedge u_2 \wedge \dots \wedge u_n \le u \text{ for some } u_i \in X \; (1 \le i \le n)\}.\]
\end{lemma}
\begin{definition}
 \cite{Atani2} A lattice $\mathscr{L}$ containing  $1$ is called an \textit{$\mathscr{L}$-domain} if for any $a, b \in \mathscr{L}$, \[ u \vee v = 1 \implies u = 1 \text{ or } v = 1.  \] 
 \end{definition}
 Consequently, $\mathscr{L}$ is an $\mathscr{L}$-domain if and only if $\{1\}$ is a prime filter of $\mathscr{L}$.

\begin{definition}
\cite{Atani2} A subset $S$ of $\mathscr{L}$ is \textit{$\vee$-closed} if  $s_1 \vee s_2 \in S$ for all $s_1, s_2 \in S$ and $0 \in S$.
\end{definition}
Let $F$ be a prime filter of $\mathscr{L}$. Then, $\mathscr{L} \setminus F$ is a $\vee$-closed subset of $\mathscr{L}$. 
\section{$S$-filters}
In this section, we introduce and investigate the  concept of $S$-filters in a bounded lattice. We initiate our study with the following definition.
\begin{definition}
Let $S \subseteq \mathscr{L}$ be   $\vee$-closed. A proper filter $\mathfrak {q}$ of $\mathscr{L}$ is called an $S$-filter if for all $u,v \in \mathscr{L}$, $u \vee v \in \mathfrak{q}$ and $u \in S$ imply $v \in \mathfrak{q}$.
\end{definition}
\begin{example}
Consider the lattice $\mathscr{L} = \{0, u, v, w, 1\}$   with the relations $0 \le u \le w \le 1$, $0 \le v \le w \le 1$, $u \vee v = w$, and $u \wedge v = 0$. In this lattice,  $S=\{0,u\}$ is a $\vee$-closed subset of $\mathscr{L}$. The filter $\mathfrak{q}_1=\{1.w,v\}$ is an $S$-filter of $\mathscr{L}$. However, the filter $\mathfrak{q}_2= \{1, w\}$  is  not an $S$-filter  of $\mathscr{L}$ since $u \vee v \in  \mathfrak{q}$ and $u \in S$ but $v \notin \mathfrak{q}$. Moreover, it is easy to check that  $\mathfrak{q}_3= \{1, w,u\}$ is  not an $S$-filter of $\mathscr{L}$.
\end{example}
\begin{remark} \label{0}
A prime filter $\mathfrak {q}$ of $\mathscr{L}$ is an $S$-filter if and only if $S \cap  \mathfrak {q}=\varnothing$.
\end{remark}
\begin{proposition} \label{1}
Let $S \subseteq \mathscr{L}$ be   $\vee$-closed and $\mathfrak {q}$ be an $S$-filter of $\mathscr{L}$. 
\begin{enumerate}
\item $\mathfrak {q} \cap S=\varnothing$.
\item $(\mathfrak {q} :_{\mathscr{L}}\mathfrak {p})$ is an $S$-filter of $\mathscr{L}$ such that $\mathfrak {p} \subseteq \mathscr{L} \setminus \mathfrak {q}$.
\end{enumerate}
\end{proposition}
\begin{proof}
(1) Suppose that $\mathfrak {q} \cap S \neq \varnothing$. There exists $u \in \mathfrak {q} \cap S$. Since $u \vee 0 \in \mathfrak {q}$ and $u \in S$, we get $0 \in \mathfrak {q}$ which is contradiction. Thus, $\mathfrak {q} \cap S=\varnothing$.

(2) Let for $u, v \in \mathscr{L}$, $u \vee u \in (\mathfrak {q} :_{\mathscr{L}}\mathfrak {p})$  and $u \in S$. Then, $u \vee v \vee w \in \mathfrak {q}$ for all $w \in \mathfrak {p}$. Since $\mathfrak {q}$ is an $S$-filter of $\mathscr{L}$, we have $v \vee w \in \mathfrak {q}$ for all $w \in \mathfrak {p}$ and so $v \in (\mathfrak {q} :_{\mathscr{L}}\mathfrak {p})$. Hence, $(\mathfrak {q} :_{\mathscr{L}}\mathfrak {p})$ is an $S$-filter of $\mathscr{L}$.
\end{proof}
\begin{theorem} \label{2}
Let $S \subseteq \mathscr{L}$ be   $\vee$-closed and $\mathfrak {q}$ be a proper filter of  $\mathscr{L}$.  The following statements are equivalent: 
\begin{enumerate} 
\item[(i)] $\mathfrak {q}$ is an $S$-filter of $\mathscr{L}$;
\item[(ii)] For any two filters $\mathfrak{r}$ and $\mathfrak{p}$  of $\mathscr{L}$, $\mathfrak{r} \vee \mathfrak{p} \subseteq \mathfrak{q}$ and $\mathfrak{r}  \cap S \neq \varnothing$ imply $\mathfrak{p} \subseteq \mathfrak{q}$. \end{enumerate}
\end{theorem}
\begin{proof}
(i) $\Longrightarrow$ (ii) Let $\mathfrak {q}$ be an $S$-filter of $\mathscr{L}$. Let us assumee that $\mathfrak{r} \vee \mathfrak{p} \subseteq \mathfrak{q}$ for filters  $\mathfrak{r}$ and $\mathfrak{p}$  of $\mathscr{L}$ and $\mathfrak{r}  \cap S \neq \varnothing$ but $\mathfrak{p} \nsubseteq \mathfrak{q}$. There exits $v \in \mathfrak{p}$ such that $v \notin \mathfrak{q}$. Since $\mathfrak{r}  \cap S \neq \varnothing$, there exists $u \in \mathfrak{r}  \cap S$. From $ u \vee v \in \mathfrak{q}$, it follows that $v \in \mathfrak{q}$ as $\mathfrak {q}$ is an $S$-filter of $\mathscr{L}$ and $u \in S$. This is a contradiction. Hence, (ii) holds.

(ii) $\Longrightarrow$ (i) Let for $u,v \in \mathscr{L}$,  $u \vee v \in \mathfrak {q}$  and $u \in S$. Put $\mathfrak{r}=T(\{u\})$ and $\mathfrak{p}=T(\{v\})$. From $u \in S$ it follows that $\mathfrak{r} \cap S \neq \varnothing$.  Since  $\mathfrak{r} \vee \mathfrak{p} \subseteq \mathfrak{q}$, we obtain $\mathfrak{p} \subseteq \mathfrak{q}$ by the assumption. This implies that $v \in \mathfrak{q}$. Thus, $\mathfrak {q}$ is an $S$-filter of $\mathscr{L}$.
\end{proof}
Let $S \subseteq \mathscr{L}$ be   $\vee$-closed and $\mathfrak{p}$ be a filter of $\mathscr{L}$. Then, we define $\mathfrak{p}_S=\{a \in \mathscr{L} \mid a \vee t \in \mathfrak{p} \  \mathrm{for\  some} \ t \in S\}$
\begin{theorem} \label{small}
Let $S \subseteq \mathscr{L}$ be   $\vee$-closed and $\mathfrak{p}$ be a filter of $\mathscr{L}$ such that $\mathfrak{p} \cap S=\varnothing$. Then, $\mathfrak{p}_S$  is the smallest $S$-filter of $\mathscr{L}$ containig $\mathfrak{p}$.
\end{theorem}
\begin{proof}
Since  $1 \vee 0 \in \mathfrak{p}$, we obtain $1 \in \mathfrak{p}_S$ and so  $\mathfrak{p}_S$ is non-empty subset of $\mathscr{L}$. Let $a, b \in \mathfrak{p}_S $. Then there exists $t_a,t_b \in S$ such that $a \vee t_a, b \vee t_b \in  \mathfrak{p} $. Therefore, we get $(a \wedge b) \vee (t_a \vee t_b)=(a \vee  t_a \vee t_b) \wedge (b \vee t_a \vee t_b) \in \mathfrak{p}$. Hence, $a \wedge b \in \mathfrak{p}_S$ as $t_a \vee t_b \in S$. Moreover, we have  $(a \vee x) \vee  t_a =(a \vee t_a) \vee x \in \mathfrak{p}$ for all $x \in \mathscr{L}$. This means that $a \vee x \in \mathfrak{p}_S$. Therefore, $\mathfrak{p}_S$ is a filter of $\mathscr{L}$. Now, let for $u,v \in   \mathscr{L}$, $u \vee v \in \mathfrak{p}_S$ and $u \in S$. Hence, we obtain $u \vee v \vee t \in \mathfrak{p}$ for some $t \in S$. Since $u \vee t \in S$, we conclude that $v \in \mathfrak{p}_S$ by the definition of $\mathfrak{p}_S$. Assume that $\mathfrak{q}$ is an $S$-filter of $\mathscr{L}$ containing $\mathfrak{p}$ and $x \in \mathfrak{p}_S$. Then, there exists $t \in S$ with $x \vee t \in \mathfrak{p}$. This implies that $x  \vee t \in \mathfrak{q}$. Since $\mathfrak{q}$ is an $S$-filter of $\mathscr{L}$ and $t \in S$, we get $x \in \mathfrak{q}$. Thus,   $\mathfrak{p}_S \subseteq  \mathfrak{q}$ which implies $\mathfrak{p}_S$ is the smallest $S$-filter containig $\mathfrak{p}$.
\end{proof}
\begin{theorem} \label{ghasem}
Let $S \subseteq \mathscr{L}$ be   $\vee$-closed and $\mathfrak {q}$ be a filter of $\mathscr{L}$. Then, the following statements are equivalent:
\begin{enumerate}
\item $\mathfrak {q}$ is an $S$-filter of $\mathscr{L}$.
\item $\mathfrak {q}=(\mathfrak {q} :_{\mathscr{L}}t)$ for each $t \in S$.
\item $\mathfrak{q}_S=\mathfrak{q}$.
\end{enumerate}
\end{theorem}
\begin{proof}
(1) $\Longrightarrow$ (2) Let $\mathfrak {q}$ be an $S$-filter of $\mathscr{L}$. For for each $t \in S$, the inclusion $\mathfrak {q} \subseteq (\mathfrak {q} :_{\mathscr{L}}t)$ always holds. Now, let $x \in (\mathfrak {q} :_{\mathscr{L}}t)$ for $t \in S$. From $x \vee t \in \mathfrak {q}$ it follows that $x \in \mathfrak {q}$ as  $\mathfrak {q}$ is an $S$-filter of $\mathscr{L}$ and $t \in S$. This implies that $(\mathfrak {q} :_{\mathscr{L}}t) \subseteq \mathfrak {q}$ and so  $\mathfrak {q}=(\mathfrak {q} :_{\mathscr{L}}t)$.

(2) $\Longrightarrow$ (3) Let $\mathfrak {q}=(\mathfrak {q} :_{\mathscr{L}}t)$ for each $t \in S$. Take any $a \in \mathfrak {q}_S$. Then, we have $a \vee s \in \mathfrak {q}$ for some $s \in S$. This means that $a \in (\mathfrak {q} :_{\mathscr{L}}s)$. By the hypothesis, we get $a \in \mathfrak {q}$ and so $\mathfrak{q}_S \subseteq \mathfrak{q}$. Since  the reverse containment is obvious, we have $\mathfrak{q}_S=\mathfrak{q}$.

(3) Let for $u,v \in \mathscr{L}$, $u \vee v \in \mathfrak {q}$ and $u \in S$. This means that $u \in \mathfrak{q}_S$ by the definition of $\mathfrak{q}_S$. Then, we have $u \in \mathfrak{q}$ by the hypothesis. Consequently, $\mathfrak {q}$ is an $S$-filter of $\mathscr{L}$.
\end{proof}
Now, we present an  $S$-filter version of the celebrated Prime Avoidance Lemma.
\begin{theorem} \label{3}
Let $S \subseteq \mathscr{L}$ be   $\vee$-closed. Let $\mathfrak {p},\mathfrak{q}_1,\ldots,  \mathfrak{q}_n$ be proper filters of $\mathscr{L}$ such that $\mathfrak{p} \subseteq \bigcup_{i=1}^n \mathfrak{q}_i$ but no $\mathfrak{q}_i$ can be eliminated from the union. If $\mathfrak{q}_1$ is an $S$-filter of $\mathscr{L}$ and $\mathfrak{q}_i \cap S \neq \varnothing$ for all $2 \leq i \leq n$, then   $\mathfrak{p} \subseteq   \mathfrak{q}_1$.
\end{theorem}
 
\begin{proof}
Let $\mathfrak{p} \subseteq \bigcup_{i=1}^n \mathfrak{q}_i$ where  $\mathfrak {p},\mathfrak{q}_1,\ldots,  \mathfrak{q}_n$ are   filters of $\mathscr{L}$ but no $\mathfrak{q}_i$ can be eliminated from the union. Then, $\mathfrak{p} \nsubseteq \bigcup_{i=2}^n \mathfrak{q}_i$. Therefore, there exists $x \in  \mathfrak{p}$ such that $x \notin \bigcup_{i=2}^n \mathfrak{q}_i$. So, we have $x \in \mathfrak{q}_1$. Let $y \in \mathfrak{p} \cap \bigcap_{i=2}^n \mathfrak{q}_i$. It follows that $x \wedge y \in  \mathfrak {p}$. If $x \wedge y \in \bigcup_{i=2}^n \mathfrak{q}_i$, then $x \wedge y \in \mathfrak{q}_j$ for some $2 \leq j \leq n$ which means $x \in \mathfrak{q}_j$. Hence, we get  $x \in \bigcup_{i=2}^n \mathfrak{q}_i$ which is impossible. Then, $x \wedge y \notin \bigcup_{i=2}^n \mathfrak{q}_i$ and so $x \wedge y \in \mathfrak{q}_1$. Then, we have  $y \in \mathfrak{q}_1$. Hence, we conclude that $\mathfrak{q} \cap \bigcap_{i=2}^n \mathfrak{q}_i \subseteq \mathfrak{q}_1$ which means $\mathfrak{p} \vee \bigvee_{i=2}^n \mathfrak{q}_i \subseteq \mathfrak{q}_1$. By the hypothesis, we have  $S \cap \bigvee_{i=2}^n \mathfrak{q}_i \neq \varnothing$. This means that $\mathfrak{p} \subseteq \mathfrak{q}_1$ by Theorem \ref{2}.
\end{proof}
\begin{proposition} \label{intersection}
Let $S \subseteq \mathscr{L}$ be   $\vee$-closed and $\{\mathfrak{q}_i\}_{i\in \Delta}$ be a nonempty set of the  $S$-filters of $\mathscr{L}$. Then, $\bigcap_{i\in \Delta} \mathfrak{q}_i$ is an $S$-filters of $\mathscr{L}$.
\end{proposition}
\begin{proof}
Let for $u,v \in \mathscr{L}$, $u \vee v \in \bigcap_{i\in \Delta} \mathfrak{q}_i$  and $u \in S$. Since $\mathfrak{q}_i$ is an $S$-filters of $\mathscr{L}$ for all $i\in \Delta$ and $u \vee v \in \mathfrak{q}_i$, we obtain $v \in \mathfrak{q}_i$ because $u \in S$. So, $v \in \bigcap_{i\in \Delta}\mathfrak{q}_i$. This implies that  $\bigcap_{i\in \Delta} \mathfrak{q}_i$ is an $S$-filters of $\mathscr{L}$.
\end{proof}
\begin{theorem} \label{khamen}
Let $S \subseteq \mathscr{L}$ be   $\vee$-closed and  $F$ be  a filter of $\mathscr{L}$  such that $F \cap S= \varnothing$. Then there exists an $S$-filter $\mathfrak{q}$ of $\mathscr{L}$ such that $F \subseteq \mathfrak{q}$.
\end{theorem}
\begin{proof}
Assume that  $S \subseteq \mathscr{L}$ is    $\vee$-closed  and $F$ is a filter of $\mathscr{L}$  such that $F \cap S= \varnothing$. Put $\Sigma = \{\mathfrak{p} \in \mathcal{F}(\mathscr{L}) \mid F \subseteq \mathfrak{p} \text{ and } S \cap \mathfrak{p}  = \emptyset\}$. Since $F \in \Sigma$,  we have $\Sigma \neq \emptyset$.  Since $(\Sigma, \subseteq)$ is a partial order, there is a filter 
$\mathfrak{q}$ which is maximal in
$\Sigma$ by Zorn's Lemma. By the proof of Lemma 2.12 in \cite{Atani2}, $\mathfrak{q}$ is prime. Since $\mathfrak{q}$ is a prime filter disjoint with $S$, we conclude that $\mathfrak{q}$ is an $S$-filter of $\mathscr{L}$ by Remark \ref{0}.
\end{proof}
Let  $F$ be a filter of $\mathscr{L}$. By  $\min(F)$, we mean  the set of all prime filters of $\mathscr{L}$ that are minimal over $F$. Moreover, an $S$-fiter $F$ of $\mathscr{L}$ is a maximal $S$-fiter if it is not properly contained in any other $S$-fiter.
\begin{theorem}
Let $S \subseteq \mathscr{L}$  be   $\vee$-closed.    Then the following hold:

\begin{enumerate} 
\item[(i)]
If $\mathfrak{q}$ is an $S$-filter of $\mathscr{L}$ and $\mathfrak{p} \in \min(\mathfrak{q})$, then $\mathfrak{p}$ is an $S$-filter of $\mathscr{L}$.
\item[(ii)] If $\mathfrak{q}$ is a maximal $S$-filter of $\mathscr{L}$, then $\mathfrak{q}$ is a prime filter of $\mathscr{L}$.
\end{enumerate} 
\end{theorem}
\begin{proof}
(i) Let     for $u,v \in \mathscr{L}$, $u \vee v \in \mathfrak{p}$ and $u \in  S$. Then, there exists $w \notin \mathfrak{p}$ such that $u \vee v \vee w \in \mathfrak{q}$ by Proposition 2.7 in \cite{Atani1}. Since $\mathfrak{q}$ is an $S$-filter of $\mathscr{L}$ and $u \in S$,  we get $v \vee w \in \mathfrak{q} \subseteq  \mathfrak{p}$. This means  that $v \in \mathfrak{p}$, as $P$ is a prime filter and $w \notin \mathfrak{p}$. Consequently, $\mathfrak{p}$ is  an $S$-filter of $\mathscr{L}$.

(ii) Assume that $\mathfrak{q}$ is a maximal $S$-filter of $\mathscr{L}$. Let    $u \vee v \in \mathfrak{q}$ for $u,v \in \mathscr{L}$  but $u \notin  \mathfrak{q}$. Therefore, $(\mathfrak{q} : \{u\})$ is an $S$-filter of $\mathscr{L}$ by Proposition \ref{1} (2). So, we have $\mathfrak{q} \subseteq (\mathfrak{q} :_{\mathscr{L}} \{u\})$. By maximality of $\mathfrak{q}$, we obtain $v \in (\mathfrak{q} :_{\mathscr{L}} \{u\})=\mathfrak{q}$. Thus, $\mathfrak{q}$ is a prime filter of $\mathscr{L}$.
\end{proof}
Let $S$ and $ S^{\prime}$ be two subsets of $\mathscr{L}$.  $S^{\prime}$ is called    $S$-$\vee$-closed  if $S^{\prime}$ contains at least an element $t \neq 0$ from $S$ and $t \vee t^{\prime} \in S^{\prime}$ for all $t \in S \cap S^{\prime}$ and $t^{\prime} \in S^{\prime}$. Moreover, $S^{\prime}$ is said to be $S$-complete  if $S^{\prime}$ is $S$-$\vee$-closed and $x \vee y \in S^{\prime}$ for all $x,y \in \mathscr{L}$ implies $x \in S^{\prime}$ and $y \in S^{\prime}$.
\begin{proposition} \label{complete}
Let $\{\mathfrak{q}_i\}_{i \in I}$ be a family of  $S$-filters of $\mathscr{L}$. Then, $S^{\prime}=\mathscr{L} \setminus \bigcup_{i \in I} \mathfrak{q}_i$ is $S$-complete.
\end{proposition}
\begin{proof}
Assume that  $\{\mathfrak{q}_i\}_{i \in I}$ are a family of $S$-filters of $\mathscr{L}$. Then, $\mathfrak{q}_i \cap S= \varnothing$ for all $i \in I$ by Proposition \ref{1} (1). This implies that $S \subseteq S^{\prime}$. Take any $t \in S \cap S^{\prime}$ and $t^{\prime} \in  S^{\prime}$. This means that $t, t^{\prime} \notin  \mathfrak{q}_i$ for all $i \in I$. Since $\mathfrak{q}_i$ is an $S$-filter for each $i \in I$, we get $t \vee t^{\prime} \notin \mathfrak{q}_i$. Therefore, we have $t \vee t^{\prime} \in S^{\prime}$ which implies $S^{\prime}$ is $S$-$\vee$-closed. Now, let $x \vee y \in S^{\prime}$ for   $x,y \in \mathscr{L}$. So, $x \vee y \notin \mathfrak{q}_i$ for every $i \in I$. It follows that for all $1 \leq i \leq t$ neither $x \in \mathfrak{q}_i$ nor $y \in \mathfrak{q}_i$. Hence, we get $x \in S^{\prime}$ and $y \in S^{\prime}$. Thus, $S^{\prime}=\mathscr{L} \setminus \bigcup_{i \in I}  \mathfrak{q}_i$ is $S$-complete.
\end{proof}
\begin{theorem}
Let $S$ and $ S^{\prime}$ be two subsets of $\mathscr{L}$. Then, $S^{\prime}$ is $S$-complete containing $S$ if and only if $S^{\prime}=\mathscr{L} \setminus \bigcup_{\mathfrak{q} \in \mathscr{P}}  \mathfrak{q}$ where  $\mathscr{P}$ is a family of  $S$-filters of $\mathscr{L}$.
\end{theorem}
\begin{proof}
$\Longrightarrow$ Let $S^{\prime}$ be  $S$-complete containing $S$. Assume that $\mathscr{P}$ is the set of all  $S$-filters of $\mathscr{L}$ disjoint $S^{\prime}$. It is obvious that $S^{\prime} \subseteq \mathscr{L} \setminus \bigcup_{\mathfrak{q} \in \mathscr{P}}  \mathfrak{q}$. Now, suppose that  $a \in \mathscr{L} \setminus \bigcup_{\mathfrak{q} \in \mathscr{P}}  \mathfrak{q}$ such that $a \notin S^{\prime}$. Let us assume that $T(\{a\}) \cap S^{\prime} \neq \varnothing$. Then, we have   $a \vee x \in S^{\prime}$ for some $ x \in \mathscr{L}$. Since $S^{\prime}$ is $S$-complete, $a \in S^{\prime}$ which is impossibe. Therefore,  $T(\{a\}) \cap S^{\prime} = \varnothing$ which implies $T(\{a\}) \cap S=\varnothing$. Hence, $T(\{a\})_S$ is an $S$-filter of $\mathscr{L}$ by Theorem \ref{ghasem}. Therefore, $a \in T(\{a\}) \subseteq \bigcup_{\mathfrak{q} \in \mathscr{P}}  \mathfrak{q}$ which  contradicts the  fact that $a \in \mathscr{L} \setminus \bigcup_{\mathfrak{q} \in \mathscr{P}}  \mathfrak{q}$. Then, we conclude that $a \in S^{\prime}$ and so $S^{\prime}=\mathscr{L} \setminus \bigcup_{\mathfrak{q} \in \mathscr{P}}  \mathfrak{q}$ where  $\mathscr{P}$ is a family of  $S$-filters of $\mathscr{L}$.

$\Longleftarrow$ It follows from Proposition \ref{complete}.
\end{proof}
Let $\mathscr{L}_1$ and $\mathscr{L}_2$ be lattices. A lattice homomorphism $\psi :\mathscr{L}_1 \rightarrow \mathscr{L}_2$ is a map from $\mathscr{L}_1$ to $\mathscr{L}_2$ satisfying $\psi(u \vee v) = \psi(u) \vee \psi(v)$ and $\psi(u \wedge v) = \psi(u) \wedge \psi(v)$ for all $u, v \in \mathscr{L}_1$  \cite{Birkhoff}. Moreover, if $\psi(1)=1$, then the set $\operatorname{Ker}(\psi)=\{u \in  \mathscr{L} \mid \psi(u)=1\}$ is a filter of $\mathscr{L}$ by Lemma 3.13 in \cite{Atani2}.
\begin{theorem} \label{homo}
Let $S \subseteq \mathscr{L}$  be   $\vee$-closed and $\psi :\mathscr{L}_1 \rightarrow \mathscr{L}_2$ be a lattice homomorphism such that $\psi(1) = 1$. The following hold:
 \begin{enumerate}
\item If $\mathfrak{q}_2$ is a $\psi(S)$-filter of  $\mathscr{L}_2$, then  $\psi^{-1}(\mathfrak{q}_2)$ is an $S$-filter of $\mathscr{L}_1$.
\item If $\mathscr{L}_1$ is   complemented, $\psi$ is onto, and $\mathfrak{q}_1$ is an $S$-filter of $\mathscr{L}_1$ containing $\operatorname{Ker}(\psi)$, then $f(\mathfrak{q}_1)$ is an $\psi(S)$-filter of $\mathscr{L}_2$.
 \end{enumerate}
\end{theorem}
\begin{proof}
(1) Let for $u,v \in \mathscr{L}_1$, $u \vee v \in \psi^{-1}(\mathfrak{q}_2)$  and $u \in S$. Therefore, we have $\psi(u \vee v)=\psi(u) \vee \psi(v) \in \mathfrak{q}_2$ with $\psi(u) \in \psi(S)$. Since $\mathfrak{q}_2$ is a $\psi(S)$-filter of  $\mathscr{L}_2$, we get $\psi(v) \in \mathfrak{q}_2$ which means $v \in \psi^{-1}(\mathfrak{q}_2)$. Hence, $\psi^{-1}(\mathfrak{q}_2)$ is an $S$-filter of $\mathscr{L}_1$.

(2) Let for $u_2,v_2 \in \mathscr{L}_2$, $u_2 \vee v_2 \in \psi(\mathfrak{q}_1)$  and $u_2 \in \psi(S)$. Then, there exists $u_1 \in S$ such that $\psi(u_1)=u_2$. Since $\psi$ is onto, there exist $v_1 \in \mathscr{L}_1$ with $\psi(v_1)=v_2$. Hence, $u_2 \vee v_2=\psi(u_1) \vee \psi(v_1)=\psi(u_1 \vee v_1) \in \psi(\mathfrak{q}_1) $. This imples that $\psi(u_1 \vee v_1)=\psi(a)$ for some $a \in \mathfrak{q}_1$. Since $\mathscr{L}_1$ is   complemented, there exists $b \in \mathscr{L}_1$ such that $a \vee b =1$ and $a \wedge b =0$. Since $\psi(u_1 \vee v_1 \vee b)=\psi(u_1 \vee v_1) \vee \psi(b)=1$, we get $u_1 \vee v_1 \vee b \in \operatorname{Ker}(\psi) \subseteq  \mathfrak{q}_1$. Since $\mathfrak{q}_1$ is a filter of $\mathscr{L}_1$, we have $u_1 \vee v_1=(u_1 \vee v_1) \vee (a \wedge b)=(u_1 \vee v_1 \vee a) \wedge (u_1 \vee v_1 \vee b) \in \mathfrak{q}_1$. Since $\mathfrak{q}_1$ is an $S$-filter of $\mathscr{L}_1$ and $u_1 \in S$, we obtain $v_1 \in \mathfrak{q}_1$ which means $v_2=\psi(v_1)  \in \psi(\mathfrak{q}_1)$. Consequently, $\psi(\mathfrak{q}_1)$ is an $f(S)$-filter of $\mathscr{L}_2$.
\end{proof}
Let  $\mathfrak{p}$ be  a filter of a lattice ($\mathscr{L}$, $\le$). Let  us define   the  relation $\sim$ on $\mathscr{L}$ as  $u \sim v$ if and only if there exist $x, y \in \mathfrak{p}$ satisfying $u \wedge x = v \wedge y$. In this case,  $\sim$ is an equivalence relation on $\mathscr{L}$.  Assume that $u \wedge \mathfrak{p}$ is the equivalence class of $u$   and $\frac{\mathscr{L}}{\mathfrak{p}}$ is the collection of all equivalence classes. Now, consider the partial order $\le_Q$ on $\frac{\mathscr{L}}{\mathfrak{p}}$ as follows: for each $u \wedge \mathfrak{p}, v \wedge \mathfrak{p} \in \frac{\mathscr{L}}{\mathfrak{p}}$,  $u \wedge \mathfrak{p} \le_Q v \wedge \mathfrak{p}$ if and only if $u \le v$. Then $(\frac{\mathscr{L}}{\mathfrak{p}}$, $\le_Q)$ is a lattice with $(u \wedge \mathfrak{p} ) \vee_Q (v \wedge \mathfrak{p} ) = (u \vee v) \wedge \mathfrak{p}$ and $(u \wedge \mathfrak{p} ) \wedge_Q (v \wedge \mathfrak{p} ) = (u \wedge v) \wedge \mathfrak{p}$ for all elements $u \wedge \mathfrak{p}, v \wedge \mathfrak{p} \in \frac{\mathscr{L}}{\mathfrak{p}}$. Note that $u \wedge \mathfrak{p} = \mathfrak{p}$ if and only if $u \in \mathfrak{p}$ \cite{Atani2}. Let $S \subseteq \mathscr{L}$ is $\vee$-closed. If we denote   the equivalence class of $u$ in $\frac{\mathscr{L}}{\mathfrak{p}}$ by $\bar{u}$, then $\bar{S}=\{\bar{s}  \mid s \in S\} \subseteq \frac{\mathscr{L}}{\mathfrak{p}}$ is $\vee$-closed. 
\begin{corollary}
Let $S \subseteq \mathscr{L}$  be   $\vee$-closed and let $\mathfrak{p}$ and $ \mathfrak{q}$ be filters of $\mathscr{L}$ such that $\mathfrak{p} \subseteq \mathfrak{q}$. If $ \mathfrak{q}$ is an $S$-filter of $\mathscr{L}$, then $\frac{\mathfrak{q}}{\mathfrak{p}}$ is an $\bar{S}$-filter of $\frac{\mathscr{L}}{\mathfrak{p}}$.
\end{corollary}
\begin{proof}
Assume that $\psi: \mathscr{L} \rightarrow \frac{\mathscr{L}}{\mathfrak{p}}$ such that   $u \mapsto u \wedge \mathfrak{p}$. By Proposition 4.4 (1) in \cite{Atani1}, $\psi$ is a natural epimorphism. Then, $\operatorname{Ker}(\psi)=\{u \mid u \wedge \mathfrak{p}=1 \wedge \mathfrak{p}\}=\mathfrak{p} \subseteq \mathfrak{q}$ by Lemma 4.3 in \cite{Atani4}. Thus,  $\psi(\mathfrak{q})=\{u \wedge \mathfrak{p} \mid u \in \mathfrak{q}\}=\frac{\mathfrak{q}}{\mathfrak{p}}$ is an $\bar{S}$-filter of $\frac{\mathscr{L}}{\mathfrak{p}}$ by Theorem \ref{homo} (2).
\end{proof}
Suppose that $(\mathscr{L}_1,\le_1),\ldots,(\mathscr{L}_t,\le_t)$ are lattices ($t \geq 2$). We define a partial order $\le_c$ on $\mathscr{L}_1 \times \cdots \times \mathscr{L}_t$ as follows: $(u_1,\ldots,u_t) \le_c (v_1,\ldots,v_t)$ for all $(u_1,\ldots,u_t), (v_1,\ldots,v_t) \in \mathscr{L}_1 \times \cdots \times \mathscr{L}_t$ if and only if $u_i \le_i v_i$ for all $1 \le i \le t$. It easy to see that $(\mathscr{L},\le_c)$ is a lattice such that $(u_1,\ldots,u_t) \vee_c (v_1,\ldots,v_t)=(u_1 \vee v_1,\cdots,u_t \vee v_t)$ and $(u_1,\ldots,u_t) \wedge_c (v_1,\ldots,v_t)=(u_1 \wedge v_1,\cdots,u_t \wedge v_t)$ for all $(u_1,\ldots,u_t), (v_1,\ldots,v_t) \in \mathscr{L}_1 \times \cdots \times \mathscr{L}_t$. In this case,  $\mathscr{L}=\mathscr{L}_1 \times \cdots \times \mathscr{L}_t$ is called a decomposable lattice \cite{Atani2}.

\begin{theorem} \label{car}
Let $\mathscr{L}=\mathscr{L}_1 \times \cdots \times \mathscr{L}_t$ be a decomposable lattice, and let $\mathfrak{q}_i \subseteq \mathscr{L}_i$ be a filter and $S_i \subseteq \mathscr{L}_i$ be a $\vee$-closed for any $1 \le i \le t$. Then, $\mathfrak{q}= \mathfrak{q}_1 \times \cdots \times  \mathfrak{q}_t$ is an $S_1 \times \cdots \times S_t$-filter of $\mathscr{L}$ if and only if $\mathfrak{q}_i$ is an $S_i$-filter of $\mathscr{L}_i$ for any $1 \le i \le t$.
\end{theorem}
\begin{proof}
$\Longrightarrow$ Let $\mathfrak{q}$ be an $S_1 \times \cdots \times S_t$-filter of $\mathscr{L}$. Take any $1 \le k \le t$. Assume that for $u_k,v_k \in \mathscr{L}_k$, $u_k \vee v_k \in \mathfrak{q}_k$  and $u_k \in S_k$. Put 
\[a = ( \underbrace{0, 0, \dots, 0, \underset{\substack{\uparrow \\ k\text{-th}}}{u_k}, 0, \dots, 0}_{t \text{ components}} )\]
and
\[b = ( \underbrace{1, 1, \dots, 1, \underset{\substack{\uparrow \\ k\text{-th}}}{v_k}, 1, \dots, 1}_{t \text{ components}} ).\]
Then, we have 
\[a \vee_c b=(1, 1, \dots, 1,  u_k \vee v_k , 1, \dots, 1) \in \mathfrak{q}.\]
Since $\mathfrak{q}$ is an $S_1 \times \cdots \times S_t$-filter of $\mathscr{L}$ and $a \in S_1 \times \cdots \times S_t$, we conclude that $b \in \mathfrak{q}$ which means $u_k \in \mathfrak{q}_k$. Thus, $\mathfrak{q}_i$ is an $S_i$-filter of $\mathscr{L}_i$ for any $1 \le i \le t$.

$\Longleftarrow$ Assume that for $(u_1,\ldots,u_t),(v_1,\ldots,v_t) \in  \mathscr{L}$, $(u_1,\ldots,u_t) \vee _c(v_1,\ldots,v_t)=(u_1 \vee v_1,\ldots,u_t \vee v_t) \in \mathfrak{q}$ and $(u_1,\ldots,u_t) \in S$. Therefore, we have $u_i \vee v_i \in \mathfrak{q}_i$ for each $1 \leq i \leq t$. Since every $ \mathfrak{q}_i$ is an $S_i$-filter and $u_i \in S_i$, we obtain $v_i \in  \mathfrak{q}_i$ for all $1 \leq i \leq t$. This means $(v_1,\ldots,v_t) \in \mathfrak{q}$. Consequently, $\mathfrak{q}= \mathfrak{q}_1 \times \cdots \times  \mathfrak{q}_t$ is an $S_1 \times \cdots \times S_t$-filter of $\mathscr{L}$.
\end{proof}

\end{document}